\documentclass[a4paper,12pt]{article}

\usepackage[T1]{fontenc}
\usepackage{amsmath, graphicx, color, amsthm, centernot, amssymb}
\usepackage[font=small,labelfont=bf,width=12.5cm]{caption}
\usepackage{algorithmic, algorithm}

\newtheorem{theorem}{Theorem}
\newtheorem{lemma}{Lemma}
\newtheorem{corollary}{Corollary}

\newtheorem{conjecture}{Conjecture}

\newtheorem{question}{Question}

\newcommand{\set}[1]{\ensuremath{\left\{#1 \right\}}}
\newcommand{\cac}{\ensuremath{\Psi}}
\newcommand{\chis}[1]{\chi_{\mathrm{st}}'(#1)}
\newcommand{\chisl}[1]{\mathrm{ch}_{\mathrm{st}}'(#1)}

\title{Note on list star edge-coloring of subcubic graphs}
\author
{
Borut Lu\v{z}ar\thanks{Faculty of Information Studies, Novo mesto, Slovenia. 
		E-Mail: \texttt{borut.luzar@gmail.com}}, \
	Martina Mockov\v{c}iakov\'{a}\thanks{European Centre of Excellence NTIS, University of West Bohemia, Pilsen, Czech Republic. \newline
		E-Mail: \texttt{mmockov@ntis.zcu.cz}}, \
	Roman Sot\'{a}k\thanks{Faculty of Science, Pavol Jozef \v Saf\'{a}rik University, Ko\v{s}ice, Slovakia.
		E-Mail: \texttt{roman.sotak@upjs.sk}}
}

\begin{document}
\maketitle

{
\abstract
{
	\textit{A star edge-coloring} of a graph is a proper edge-coloring without bichromatic paths and cycles of length four. 	
	In this paper, we consider the list version of this coloring and prove that the list star chromatic index of every subcubic graph
	is at most $7$, answering the question of Dvo\v{r}\'{a}k et al. 
	(Star chromatic index, J. Graph Theory 72 (2013), 313--326).
}

\bigskip
{\noindent\small \textbf{Keywords:} List star edge-coloring, star chromatic index, list star chromatic index, subcubic graph.}
}

\section{Introduction}

In this note, we consider graphs without loops but possible parallel edges (hence multigraphs). A $3$-regular graph is called \textit{cubic}, and
a graph with maximum degree $3$ is called \textit{subcubic}.
A proper edge-coloring of a graph $G$ is called a \textit{star edge-coloring} 
if there is neither bichromatic path nor bichromatic cycle of length four. The minimum number of colors for which $G$ admits a star edge-coloring
is called the \textit{star chromatic index} and it is denoted by $\chis{G}$.

The star edge-coloring was defined in 2008 by Liu and Deng~\cite{LiuDen08}, motivated by the vertex version 
(see e.g.~\cite{AlbChaKieKunRam04,BuCraMonRasWan09,CheRasWan13,Gru73}).
In 2013, Dvo\v{r}\'{a}k, Mohar and \v{S}\'{a}mal \cite{DvoMohSam13} determined (currently the best) upper and lower bounds for complete graphs. 
\begin{theorem}[Dvo\v{r}\'{a}k, Mohar, \v{S}\'{a}mal, 2013]
	\label{thm:complete}
	The star chromatic index of the complete graph $K_n$ satisfies
	$$
		2n (1 + o(1)) \le \chis{K_n} \le n \frac{2^{2\sqrt{2}(1+o(1))\sqrt{\log{n}}}}{(\log{n})^{1/4}}\,.
	$$
	In particular, for every $\epsilon >0$ there exists a constant $c$ such that $\chis{K_n} \le cn^{1+\epsilon}$ for every $n\ge 1$.
\end{theorem}

From Theorem~\ref{thm:complete}, they also derived a near-linear upper bound in terms of the maximum degree $\Delta$ for general graphs.
%\begin{theorem}[Dvo\v{r}\'{a}k, Mohar, \v{S}\'{a}mal, 2013]
%	Let $G$ be an arbitrary graph of maximum degree $\Delta$. Then
%	$$
		%\chis{G} \le \chis{K_{\Delta + 1}} \cdot O \bigg( \frac{\log \Delta}{\log \log \Delta} \bigg)^2\,,
%	$$
%	and therefore $\chis{G} \le \Delta \cdot 2^{O(1)\sqrt{\log \Delta}}$.
%\end{theorem}
In addition, they considered subcubic graphs and showed that their star chromatic index is at most $7$. 
\begin{theorem}[Dvo\v{r}\'{a}k, Mohar, \v{S}\'{a}mal, 2013]~
	\begin{itemize}
		\item[$(a)$] If $G$ is a subcubic graph, then $\chis{G} \le 7$.
		\item[$(b)$] If $G$ is a simple cubic graph, then $\chis{G} \ge 4$, and the equality holds
			if and only if $G$ covers the graph of the $3$-cube.
	\end{itemize}
\end{theorem}

Regarding subcubic graphs, Bezegov\'{a} et al.~\cite{BezLuzMocSotSkr16} proved that $4$ and $5$ colors are enough for 
every subcubic tree and subcubic outerplanar graph, respectively, and the bound for both is tight.

While there exist subcubic graphs with the star chromatic index equal to $6$, e.g. $K_4$ with one subdivided edge,
$K_{3,3}$, and the complement of $C_6$
(see Fig.~\ref{fig:6colors}), 
no example of a subcubic graph requiring $7$ colors is known.
\begin{figure}[ht!]
	$$
		\includegraphics{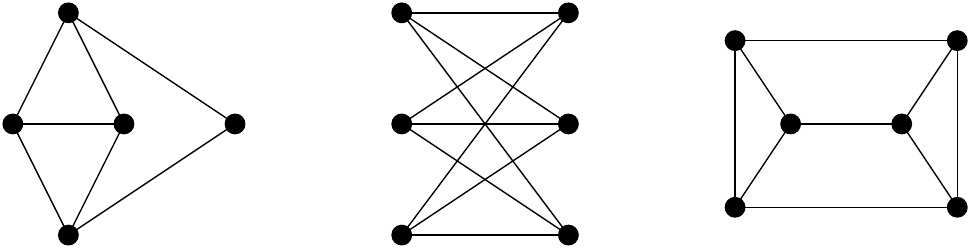}
	$$
	\caption{Subcubic graphs with star chromatic index equal to $6$.}
	\label{fig:6colors}
\end{figure}
Based on this fact, Dvo\v{r}\'{a}k et al. proposed the following conjecture.
\begin{conjecture}[Dvo\v{r}\'{a}k, Mohar, \v{S}\'{a}mal, 2013]
	\label{conj:cub}
	If $G$ is a subcubic graph, then $\chis{G}~\le~6$.
\end{conjecture}

In~\cite{DvoMohSam13} the authors suggested to study the list version of star edge-coloring.
We say that $L$ is an \textit{edge-list-assignment} for the graph $G$ if it assigns a list $L(e)$ of possible colors to each edge $e$ of $G$. 
If each of the lists is of size at most $k$, we call $L$ a \textit{$k$-edge-list-assignment}.
If $G$ admits a star edge-coloring $\sigma$ such that $\sigma(e) \in L(e)$ for all edges in $E(G)$, 
then we say that $G$ is \textit{star $L$-edge-colorable} or $\sigma$ is a {\em star $L$-edge-coloring} of $G$. 
The graph $G$ is \textit{star $k$-edge-choosable} if it is star $L$-edge-colorable for every edge-list-assignment $L$, 
where $|L(e)| \ge k$ for every $e \in E(G)$. 
The \textit{list star chromatic index} $\chisl{G}$ of $G$ is the minimum $k$ such that $G$ is star $k$-edge-choosable. 

The current paper has been motivated by the following question of Dvo\v{r}\'{a}k et al.
\begin{question}[Question 3 in~\cite{DvoMohSam13}]
	\label{q:main}
	Is it true that $\chisl{G} \le 7$ for every subcubic graph $G$? (Perhaps even $\le6$?)
\end{question}
Currently best reported upper bound is due to Kerdjoudj, Kostochka, and Raspaud~\cite{KerKosRas16} set at $8$ colors.
These authors also derived the bounds on the list star chromatic index of subcubic graphs in terms of the maximum average degree $\rm{mad}(G)$.
In this paper, we answer Question~\ref{q:main} in affirmative, i.e. we prove that the list star chromatic index of such graphs is at most 7.
\begin{theorem}
	\label{thm:main}
	Let $G$ be a graph with maximum degree $3$. Then
	$$
		\chisl{G} \le 7\,.
	$$
\end{theorem}

\section{Proof of Theorem~\ref{thm:main}}

We prove Theorem~\ref{thm:main} through several steps. 
Throughout this section, by the \textit{distance} between two edges, we mean the distance 
between the corresponding vertices in the line graph. Hence, two adjacent edges are at 
distance $1$.

First, we prove two results on star edge-choosability of cycles, then we show that the statement holds for $2$-connected
subcubic graphs, and finally, we give a proof of the main theorem.
\begin{lemma}
	\label{lem:2}
	Every cycle $C$, except $C_5$, is star $3$-edge-choosable.	
\end{lemma}

\begin{proof}
	Let $C = e_1e_2\dots e_n$ be an $n$-cycle for some $n \neq 5$, and let $L$ be a $3$-list-assignment of $C$.
	It is an easy exercise to verify that $C$ is star $3$-edge-colorable for every $n \neq 5$, 
	thus we assume that the lists of at least two edges are different,
	and hence there are adjacent edges with different lists, say $L(e_1) \neq L(e_n)$. 
	We will construct a star $3$-edge-coloring $\sigma$ with the colors from $L$.
	First, let $f(e_1) = \alpha$, where $\alpha \in L(e_1) \setminus L(e_n)$.
	Now consider two cases: 
	\begin{itemize}
		\item[$(a)$] If $n$ is even, we choose $\sigma(e_{2i+1}) \neq \sigma(e_{2i-1})$, for $i=1,\ldots,\frac{n-2}{2}$. 
			Additionally, we choose $\sigma(e_{n-1}) \neq \sigma(e_1)$, which is always possible. 			
			Then, we choose $\sigma(e_{2i})$ distinct from the colors of adjacent edges, for $i = 1,\ldots,\frac{n}{2}$. 
			This clearly induces a star edge-coloring of $C$.
		\item[$(b)$] If $n$ is odd, we similarly color $\sigma(e_{2i+1}) \neq \sigma(e_{2i-1})$, for $i=1,\ldots,\frac{n-3}{2}$. 
			Then, we choose $\sigma(e_2)$ distinct from the colors of adjacent edges. 
			Now, we can choose $\sigma(e_n)$ distinct from $\sigma(e_{n-2})$ and $\sigma(e_2)$ ($\sigma(e_n) \neq \sigma(e_1)$ via the choice of $\sigma(e_1)$). 
			Finally, we choose $\sigma(e_{2i})$ distinct from the colors of adjacent edges, for $i=2,\ldots,\frac{n-1}{2}$.
			Again, one can observe that we have used at least $3$ colors for each path on $4$ edges. 
	\end{itemize}	
\end{proof}

\begin{lemma}
	\label{lem:3}
	A $5$-cycle $C = e_1e_2e_3e_4e_5$ is star $L$-edge-colorable if for every $i\in \set{1,\dots,5}$ it holds $|L(e_i)| \ge 3$ and
	$$
		\left|\bigcup_{i=1}^5 L(e_i) \right| \ge 4\,.
	$$
\end{lemma}

\begin{proof}
	The union of the lists of at least two edges contains at least $4$ colors, so we may simply repeat the coloring procedure from the case $(b)$
	in the proof of Lemma~\ref{lem:2} to obtain a coloring of $C$.
\end{proof}

We define the \textit{supergraph} $\widehat{G}$ of a $2$-connected subcubic graph $G$ 
as the graph obtained by introducing edges between pairs of vertices of degree $2$ in $G$
until $G$ contains at most one vertex of degree $2$. 
In what follows, we will use a decomposition of $\widehat{G}$ guaranteed by the well-known result of Petersen~\cite{Pet1891}.
\begin{theorem}[Petersen, 1891]
	\label{lem:pet}
	Let $G$ be a subcubic graph without bridges and with at most one vertex of degree at most $2$. 
	Then, there is a matching $M$ that covers all vertices of degree $3$.
\end{theorem}
Notice that since there is always an even number of odd vertices in a graph, an eventual vertex $v$ of degree $2$ from the above theorem 
is never covered by $M$ and hence, as an immediate corollary of the above we have that $\widehat{G}$ contains a $2$-factor $F$
where, clearly, $v$ is a part of a cycle.
Connecting the cycles of $F$ with the minimum number of edges of $M$ such that the resulting graph is connected,
we obtain a spanning cactus $\cac$ of $\widehat{G}$.
Recall that a \textit{cactus} is a connected graph $H$ such that any two cycles in $H$ have at most one vertex in common,
and hence, in the case of subcubic graphs, any two cycles in $H$ are vertex-disjoint.
\begin{corollary}
	\label{cor:cac}
	A supergraph $\widehat{G}$ of any $2$-connected subcubic graph $G$ can be decomposed into 
	a spanning cactus $\cac$ and a matching $M'$ such that every vertex of $\widehat{G}$ is contained in
	some cycle of $\cac$.
\end{corollary}

Now, we prove the case of $2$-connected graphs.

\begin{lemma}
	\label{thm:2con}
	Every $2$-connected subcubic graph is star $7$-edge-choosable. 
\end{lemma}

\begin{proof}
	Let $G$ be a $2$-connected subcubic graph. 	
	We may assume $G$ has at least $6$ vertices, otherwise $G$ has at most $7$ edges and the result trivially follows. 
	We will prove the theorem by presenting a construction of a star $7$-edge-coloring of $\widehat{G}$ using a decomposition 
	into a spanning cactus and a matching. Clearly, a star edge-coloring of $\widehat{G}$ induces a star edge-coloring of $G$.

	By Corollary~\ref{cor:cac}, $\widehat{G}$ contains a spanning cactus $\cac$ such that every vertex of $\widehat{G}$ is 
	contained in some cycle of $\cac$, and the edges of $\widehat{G}$ not in $\cac$ form a matching $M'$, i.e. $M' = E(\widehat{G})-E(\cac)$.
	Observe that for any edge $e$ of $\widehat{G}$ there are at most four edges of $M'$ at distance at most $2$.	
	
	Let $L$ be a $7$-edge-list-assignment for the edges of $\widehat{G}$.	
	Color the edges of $M'$ in such a way that the edges at distance at most $2$ receive distinct colors. 
	Since for every edge $e \in M$ there are at most $4$ edges from $M'$ at distance at most $2$ from $e$ this is always possible using e.g. the greedy method.		
	Now, for every edge $e$ of $\cac$ remove the colors of the edges of $M'$ at distance at most $2$ from $e$ from the list $L(e)$. 	
	Note that after removal each non-colored edge retains a list of size at least $7 - m(e) \ge 3$, where $m(e)$ denotes the number of edges of $M'$
	within distance $2$ from $e$.
	
	It remains to color the edges of $\cac$. The cactus $\cac$ is comprised of cycles connected by single edges; we call them \textit{connectors}.
	Recall that each connector has a list of at least $3$ available colors. We will first color the connectors one by one in such a way that
	after removing the colors of connectors from the non-colored edges of $\cac$ at distance at most $2$ from them, 
	the union of available colors on the edges of every $5$-cycle of $\cac$ will be of size at least $4$. 
	Observe also that every non-colored edge of $\cac$ retains a list of at least $3$ available colors.
	
	Such a coloring of connectors is possible. Suppose, to the contrary, that there is a $5$-cycle $C$ in $\cac$ such that 
	all the incident edges of $C$ in $\widehat{G}$ (the edges of $M'$ or connectors) except one connector $e$ have already been colored (see Figure~\ref{fig:5}).
	\begin{figure}[htp!]
		$$
			\includegraphics{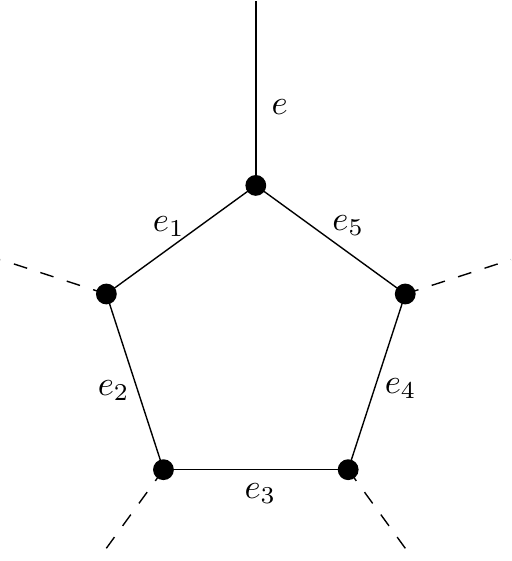}
		$$
		\caption{A $5$-cycle of $\cac$ with one yet non-colored connector $e$ and the edges of $M'$ and colored connectors depicted dashed.}
		\label{fig:5}
	\end{figure}
	While the edge $e_3$ might have only three available colors, the rest four edges of $C$ have lists of size at least $4$.
	However, in order to have all the lists the same after coloring $e$, we may assume that $L(e_i) = \set{\alpha_1,\alpha_2,\alpha_3,\alpha_4}$
	for $i \in \set{1,2,4,5}$ and $L(e_3) = \set{\alpha_1,\alpha_2,\alpha_3}$. Furthermore, the color $\alpha_4$ must be used to color $e$.
	Since $e$ has at least three available colors, we are always able to select a color for it in such a way that the $5$-cycles of $\cac$ incident with $e$ 
	(every $e$ is incident with at most two such $5$-cycles)
	will retain lists of available colors whose union will have cardinality at least $4$.
	
	Finally, we color the cycles of $\cac$. By Lemmas~\ref{lem:2} and~\ref{lem:3} this is always possible.
	Note that the given coloring of $\widehat{G}$ is indeed a star edge-coloring. 
	The edges of $M'$ and the connectors always receive colors which do not appear on the edges of
	cactus cycles at distance at most $2$ from them, hence they cannot appear in any bichromatic path or cycle of length $4$.
\end{proof}

Now, we are ready to prove the main theorem.

\begin{proof}[Proof of Theorem~\ref{thm:main}]
	We prove the theorem by contradiction. 
	Let $G$ be a subcubic graph with minimal number of vertices such that $G$ is not star $7$-edge-choosable
	for some $7$-edge-list-assignment $L$.
	Let $\widehat{G}$ be a supergraph of $G$. 
	By Lemma~\ref{thm:2con}, $\widehat{G}$ contains at least one bridge and hence at least two blocks.
	
	Since there is at most one vertex of degree $2$ in $\widehat{G}$, it contains a pendant block $B$ such that 
	$d_{\widehat{G}}(x) = 3$ for every vertex $x$ of $B$. 
	Let $e = uv$ be a bridge of $\widehat{G}$ such that $B$ is the component of $\widehat{G} - uv$ containing $v$, and $\widehat{G}_u$ 
	is the component of $\widehat{G} - uv$ containing $u$. 
	By the minimality of $G$, $\widehat{G}_u$ is star $7$-edge-choosable. 
	Let $\sigma_u$ be some coloring of $\widehat{G}_u$. 
	Color $e$ with a color from $L(e)$ that does not appear in $\sigma_u$ on the (at most $6$) edges at distance at most $2$ from $e$ in $\widehat{G}_u$.
	
	It remains to color $B$. Decompose $B$ into a spanning cactus $\cac$ and a matching $M'$. 
	Since $v$ is the only vertex of degree $2$ in $B$, both edges incident with $v$ in $B$ are the edges of some cycle of $\cac$. 
	Now, remove the color of $e$ from the lists of the edges of $B$ at distance at most $2$ from $e$ 
	(i.e. from the edges incident with $v$ and the edges adjacent to them). 
		
	Finally, as in the proof of Lemma~\ref{thm:2con}, we proceed by first coloring the edges of $M'$ greedily, 
	and then by coloring the edges of $\cac$ as described in the proof of Lemma~\ref{thm:2con}.
	Note that, since the color of $e$ is different from the colors of all the edges of $\widehat{G}$ within distance $2$,
	the edge $e$ cannot appear in any bichromatic path of length $4$. 	
	Therefore, the coloring of $B$ is a star edge-coloring, which together with $\sigma_u$ and the color of $e$ induces
	a star edge-coloring of $\widehat{G}$ and hence also of $G$. This completes the proof.
\end{proof}

\section{Discussion}

To be precise, in this paper, we only answer the first part of the question asked by Dvo\v{r}\'{a}k et al., 
while the follow up question if $6$ colors suffice for star edge-choosability of subcubic graphs remains widely open. 
In fact, it is not even known yet if subcubic planar graphs are star $6$-edge-colorable. 

On the other hand, the construction of a coloring used in the proof of Theorem~\ref{thm:main} 
can give some partial results for star $6$-edge-colorability.
In particular, if for a subcubic graph $G$ we have a decomposition into a matching $M'$ 
which can be colored with at most $3$ colors, and a subgraph with maximum degree $2$ without $5$-cycles
(and therefore also star $3$-edge-colorable), then $G$ is star $6$-edge-colorable. 

%This remark confirms Conjecture~\ref{conj:cub},
%e.g. for a subclass of planar graphs.
%\begin{proposition}
%	Let $G$ be a planar subcubic graph with girth at least $7$. Then
%	$$
		%\chis{G} \le 6\,.
%	$$
%\end{proposition}
%
%\begin{proof}
%	The graph whose vertices correspond to the edges of $M'$ and connectors of $\Psi$, where two of its
%	vertices are connected if the corresponding edges are within distance $2$ in $\widehat{G}$, 
%	is a triangle-free planar graph after removal of the edges not appearing in $G$. Hence, the edges of $M'$ 
%	can be colored with $3$ colors by Gr\"{o}tzsch's Theorem.
%	Furthermore, since the girth of $G$ is at least $7$, every cycle of $\Psi$ is star $3$-edge-colorable. 
%\end{proof}
%The above statement obviously does not hold in the list setting, since one cannot guarantee
%distinct three colors for the matching and connector edges, and other three for the edges of cycles.

However, now the upper bound for star chromatic index and list star chromatic index of subcubic graphs are the same,
which might be seen as a step toward solving another question of Dvo\v{r}\'{a}k et al.~\cite{DvoMohSam13}:
\begin{question}
	Is it true that $\chisl{G} = \chis{G}$ for every graph $G$?
\end{question}

\paragraph{Acknowledgment.} 

The first author was partly supported by the Slovenian Research Agency Program P1--0383.
The second author was partially supported by project GA17--04611S of the Czech Science Foundation and by project LO1506 of the Czech Ministry of Education, Youth and Sports.
The third author was supported by the Slovak Research and Development Agency under the Contract No. APVV--15--0116 and by the Slovak VEGA Grant 1/0368/16.

\bibliographystyle{plain}
{\small
	\bibliography{MainBase}
}

\end{document}